\documentclass[reqno, 12pt]{amsart}
\pdfoutput=1
\makeatletter
\let\origsection=\section \def\section{\@ifstar{\origsection*}{\mysection}}
\def\mysection{\@startsection{section}{1}\z@{.7\linespacing\@plus\linespacing}{.5\linespacing}{\normalfont\scshape\centering\S}}
\makeatother

\usepackage{amsmath,amssymb,amsthm}
\usepackage{mathrsfs}
\usepackage{mathabx}\changenotsign
\usepackage{dsfont}
\usepackage{bbm}

\usepackage{xcolor}
\usepackage[backref=section]{hyperref}
\usepackage[ocgcolorlinks]{ocgx2}
\hypersetup{
	colorlinks=true,
	linkcolor={red!60!black},
	citecolor={green!60!black},
	urlcolor={blue!60!black},
}

\usepackage[open,openlevel=2,atend]{bookmark}

\usepackage[abbrev,msc-links,backrefs]{amsrefs}
\usepackage{doi}

\renewcommand{\PrintDOI}[1]{\doi{#1}}

\usepackage[T1]{fontenc}
\usepackage{lmodern}
\usepackage[babel]{microtype}
\usepackage[english]{babel}

\usepackage{geometry}
\numberwithin{equation}{section}
\numberwithin{figure}{section}

\usepackage{enumitem}

\let\polishlcross=\l
\def\l{\ifmmode\ell\else\polishlcross\fi}

\let\emptyset=\varnothing
\let\setminus=\smallsetminus

\makeatletter
\def\moverlay{\mathpalette\mov@rlay}
\def\mov@rlay#1#2{\leavevmode\vtop{   \baselineskip\z@skip \lineskiplimit-\maxdimen
		\ialign{\hfil$\m@th#1##$\hfil\cr#2\crcr}}}
\newcommand{\charfusion}[3][\mathord]{
	#1{\ifx#1\mathop\vphantom{#2}\fi
		\mathpalette\mov@rlay{#2\cr#3}
	}
	\ifx#1\mathop\expandafter\displaylimits\fi}
\makeatother

\newcommand{\dcup}{\charfusion[\mathbin]{\cup}{\cdot}}

\DeclareFontFamily{U}  {MnSymbolC}{}
\DeclareSymbolFont{MnSyC}         {U}  {MnSymbolC}{m}{n}
\DeclareFontShape{U}{MnSymbolC}{m}{n}{
	<-6>  MnSymbolC5
	<6-7>  MnSymbolC6
	<7-8>  MnSymbolC7
	<8-9>  MnSymbolC8
	<9-10> MnSymbolC9
	<10-12> MnSymbolC10
	<12->   MnSymbolC12}{}
\DeclareMathSymbol{\powerset}{\mathord}{MnSyC}{180}

\usepackage{tikz}
\usetikzlibrary{calc,decorations.pathmorphing}
\usetikzlibrary{arrows,decorations.pathreplacing}
\pgfdeclarelayer{background}
\pgfdeclarelayer{foreground}
\pgfdeclarelayer{front}
\pgfsetlayers{background,main,foreground,front}

\usepackage{multicol}
\usepackage{subcaption}
\newcommand{\pedge}[9]{
	
	\ifx\relax#6\relax
	\def\qoffs{0pt}
	\else
	\def\qoffs{#6}
	\fi
	
	\def\phedge{
		($#1+#5!\qoffs!-90:#2-#5$) -- 
		($#2+#1!\qoffs!-90:#3-#1$) -- 
		($#3+#2!\qoffs!-90:#4-#2$) -- 
		($#4+#3!\qoffs!-90:#5-#3$) -- 
		($#5+#4!\qoffs!-90:#1-#4$) -- cycle}

	\coordinate (12) at ($#1!\qoffs!90:#2$);
	\coordinate (15) at ($#1!\qoffs!-90:#5$);
	\coordinate (23) at ($#2!\qoffs!90:#3$);
	\coordinate (21) at ($#2!\qoffs!-90:#1$);
	\coordinate (34) at ($#3!\qoffs!90:#4$);
	\coordinate (32) at ($#3!\qoffs!-90:#2$);
	\coordinate (45) at ($#4!\qoffs!90:#5$);
	\coordinate (43) at ($#4!\qoffs!-90:#3$);
	\coordinate (51) at ($#5!\qoffs!90:#1$);
	\coordinate (54) at ($#5!\qoffs!-90:#4$);

	\def\nphedge{
		(15) let \p1=($(15)-#1$), \p2=($(12)-#1$) in 
		arc[start angle={atan2(\y1,\x1)}, delta angle={atan2(\y2,\x2)-atan2(\y1,\x1)-360*(atan2(\y2,\x2)-atan2(\y1,\x1)>0)}, x radius=\qoffs, y radius=\qoffs] --
		(21) let \p1=($(21)-#2$), \p2=($(23)-#2$) in 
		arc[start angle={atan2(\y1,\x1)}, delta angle={atan2(\y2,\x2)-atan2(\y1,\x1)-360*(atan2(\y2,\x2)-atan2(\y1,\x1)>0)}, x radius=\qoffs, y radius=\qoffs] --
		(32) let \p1=($(32)-#3$), \p2=($(34)-#3$) in 
		arc[start angle={atan2(\y1,\x1)}, delta angle={atan2(\y2,\x2)-atan2(\y1,\x1)-360*(atan2(\y2,\x2)-atan2(\y1,\x1)>0)}, x radius=\qoffs, y radius=\qoffs] --
		(43) let \p1=($(43)-#4$), \p2=($(45)-#4$) in 
		arc[start angle={atan2(\y1,\x1)}, delta angle={atan2(\y2,\x2)-atan2(\y1,\x1)-360*(atan2(\y2,\x2)-atan2(\y1,\x1)>0)}, x radius=\qoffs, y radius=\qoffs] --
		(54) let \p1=($(54)-#5$), \p2=($(51)-#5$) in 
		arc[start angle={atan2(\y1,\x1)}, delta angle={atan2(\y2,\x2)-atan2(\y1,\x1)-360*(atan2(\y2,\x2)-atan2(\y1,\x1)>0)}, x radius=\qoffs, y radius=\qoffs] --
		cycle}

	\ifx\relax#7\relax
	\def\plwidth{1pt}
	\else
	\def\plwidth{#7}
	\fi
	
	\ifx\relax#9\relax
	\fill \nphedge;
	\else
	\fill[#9]\nphedge;
	\fi
	
	\ifx\relax#8\relax
	\draw[line width=\plwidth,rounded corners=\qoffs]\nphedge;
	\else
	\draw[line width=\plwidth,#8]\nphedge;
	\fi
}

\newcommand{\qedge}[7]{
	
	\ifx\relax#4\relax
	\def\qoffs{0pt}
	\else
	\def\qoffs{#4}
	\fi
	
	\def\qhedge{
		($#1+#3!\qoffs!-90:#2-#3$) --
		($#2+#1!\qoffs!-90:#3-#1$) --
		($#3+#2!\qoffs!-90:#1-#2$) -- cycle}

	\coordinate (12) at ($#1!\qoffs!90:#2$);
	\coordinate (13) at ($#1!\qoffs!-90:#3$);
	\coordinate (23) at ($#2!\qoffs!90:#3$);
	\coordinate (21) at ($#2!\qoffs!-90:#1$);
	\coordinate (31) at ($#3!\qoffs!90:#1$);
	\coordinate (32) at ($#3!\qoffs!-90:#2$);
	
	\def\nqhedge{
		(13) let \p1=($(13)-#1$), \p2=($(12)-#1$) in
		arc[start angle={atan2(\y1,\x1)}, delta angle={atan2(\y2,\x2)-atan2(\y1,\x1)-360*(atan2(\y2,\x2)-atan2(\y1,\x1)>0)}, x radius=\qoffs, y radius=\qoffs] --
		(21) let \p1=($(21)-#2$), \p2=($(23)-#2$) in
		arc[start angle={atan2(\y1,\x1)}, delta angle={atan2(\y2,\x2)-atan2(\y1,\x1)-360*(atan2(\y2,\x2)-atan2(\y1,\x1)>0)}, x radius=\qoffs, y radius=\qoffs] --
		(32) let \p1=($(32)-#3$), \p2=($(31)-#3$) in
		arc[start angle={atan2(\y1,\x1)}, delta angle={atan2(\y2,\x2)-atan2(\y1,\x1)-360*(atan2(\y2,\x2)-atan2(\y1,\x1)>0)}, x radius=\qoffs, y radius=\qoffs] --
		cycle}
	
	\ifx\relax#5\relax
	\def\qlwidth{1pt}
	\else
	\def\qlwidth{#5}
	\fi
	
	\ifx\relax#7\relax
	\fill \nqhedge;
	\else
	\fill[#7]\nqhedge;
	\fi
	
	\ifx\relax#6\relax
	\draw[line width=\qlwidth,rounded corners=\qoffs]\nqhedge;
	\else
	\draw[line width=\qlwidth,#6]\nqhedge;
	\fi
}

\newcommand{\redge}[8]{
	
	\ifx\relax#5\relax
	\def\qoffs{0pt}
	\else
	\def\qoffs{#5}
	\fi
	
	\def\rhedge{
		($#1+#4!\qoffs!-90:#2-#4$) -- 
		($#2+#1!\qoffs!-90:#3-#1$) -- 
		($#3+#2!\qoffs!-90:#4-#2$) -- 
		($#4+#3!\qoffs!-90:#1-#3$) -- cycle}

	\coordinate (12) at ($#1!\qoffs!90:#2$);
	\coordinate (14) at ($#1!\qoffs!-90:#4$);
	\coordinate (23) at ($#2!\qoffs!90:#3$);
	\coordinate (21) at ($#2!\qoffs!-90:#1$);
	\coordinate (34) at ($#3!\qoffs!90:#4$);
	\coordinate (32) at ($#3!\qoffs!-90:#2$);
	\coordinate (41) at ($#4!\qoffs!90:#1$);
	\coordinate (43) at ($#4!\qoffs!-90:#3$);
	
	\def\nrhedge{
		(14) let \p1=($(14)-#1$), \p2=($(12)-#1$) in 
		arc[start angle={atan2(\y1,\x1)}, delta angle={atan2(\y2,\x2)-atan2(\y1,\x1)-360*(atan2(\y2,\x2)-atan2(\y1,\x1)>0)}, x radius=\qoffs, y radius=\qoffs] --
		(21) let \p1=($(21)-#2$), \p2=($(23)-#2$) in 
		arc[start angle={atan2(\y1,\x1)}, delta angle={atan2(\y2,\x2)-atan2(\y1,\x1)-360*(atan2(\y2,\x2)-atan2(\y1,\x1)>0)}, x radius=\qoffs, y radius=\qoffs] --
		(32) let \p1=($(32)-#3$), \p2=($(34)-#3$) in 
		arc[start angle={atan2(\y1,\x1)}, delta angle={atan2(\y2,\x2)-atan2(\y1,\x1)-360*(atan2(\y2,\x2)-atan2(\y1,\x1)>0)}, x radius=\qoffs, y radius=\qoffs] --
		(43) let \p1=($(43)-#4$), \p2=($(41)-#4$) in 
		arc[start angle={atan2(\y1,\x1)}, delta angle={atan2(\y2,\x2)-atan2(\y1,\x1)-360*(atan2(\y2,\x2)-atan2(\y1,\x1)>0)}, x radius=\qoffs, y radius=\qoffs] --
		cycle}
	
	\ifx\relax#6\relax
	\def\rlwidth{1pt}
	\else
	\def\rlwidth{#6}
	\fi
	
	\ifx\relax#8\relax
	\fill \nrhedge;
	\else
	\fill[#8]\nrhedge;
	\fi
	
	\ifx\relax#7\relax
	\draw[line width=\rlwidth,rounded corners=\qoffs]\nrhedge;
	\else
	\draw[line width=\rlwidth,#7]\nrhedge;
	\fi
}

\let\epsilon=\varepsilon
\let\eps=\epsilon
\let\rho=\varrho
\let\theta=\vartheta

\newcommand{\cF}{\mathcal{F}}

\newcommand{\cM}{\mathcal{M}}

\newcommand{\ccP}{\mathscr{P}}

\newtheoremstyle{note}  {4pt}  {4pt}  {\sl}  {}  {\bfseries}  {.}  {.5em}          {}
\newtheoremstyle{introthms}  {3pt}  {3pt}  {\itshape}  {}  {\bfseries}  {.}  {.5em}          {\thmnote{#3}}
\newtheoremstyle{remark}  {2pt}  {2pt}  {\rm}  {}  {\bfseries}  {.}  {.3em}          {}

\theoremstyle{plain}
\newtheorem{theorem}{Theorem}[section]
\newtheorem{lemma}[theorem]{Lemma}
\newtheorem{prop}[theorem]{Proposition}

\newtheorem{cor}[theorem]{Corollary}

\newtheorem{claim}[theorem]{Claim}

\theoremstyle{note}

\theoremstyle{remark}
\newtheorem{remark}[theorem]{Remark}

\newtheorem{exmpl}[theorem]{Example}
\newtheorem{question}[theorem]{Question}
\newtheorem{observation}[theorem]{Observation}
\newtheorem{problem}[theorem]{Problem}

\usepackage{lineno}
\newcommand*\patchAmsMathEnvironmentForLineno[1]{
	\expandafter\let\csname old#1\expandafter\endcsname\csname #1\endcsname
	\expandafter\let\csname oldend#1\expandafter\endcsname\csname end#1\endcsname
	\renewenvironment{#1}
	{\linenomath\csname old#1\endcsname}
	{\csname oldend#1\endcsname\endlinenomath}}
\newcommand*\patchBothAmsMathEnvironmentsForLineno[1]{
	\patchAmsMathEnvironmentForLineno{#1}
	\patchAmsMathEnvironmentForLineno{#1*}}
\AtBeginDocument{
	\patchBothAmsMathEnvironmentsForLineno{equation}
	\patchBothAmsMathEnvironmentsForLineno{align}
	\patchBothAmsMathEnvironmentsForLineno{flalign}
	\patchBothAmsMathEnvironmentsForLineno{alignat}
	\patchBothAmsMathEnvironmentsForLineno{gather}
	\patchBothAmsMathEnvironmentsForLineno{multline}
}

\def\ex{\text{\rm ex}}

\usepackage{scalerel}

\makeatletter
\newcommand{\overrighharpoonup}[1]{\ThisStyle{%
		\vbox {\m@th\ialign{##\crcr
				\rightharpoonupfill \crcr
				\noalign{\kern-\p@\nointerlineskip}
				$\hfil\SavedStyle#1\hfil$\crcr}}}}

\def\rightharpoonupfill{%
	$\SavedStyle\m@th\mkern+0.8mu\cleaders\hbox{$\shortbar\mkern-4mu$}\hfill\rightharpoonuptip\mkern+0.8mu$}

\def\rightharpoonuptip{%
	\raisebox{\z@}[2pt][1pt]{\scalebox{0.55}{$\SavedStyle\rightharpoonup$}}}

\def\shortbar{%
	\smash{\scalebox{0.55}{$\SavedStyle\relbar$}}}
\makeatother

\makeatletter
\newcommand{\overlefharpoonup}[1]{\ThisStyle{%
		\vbox {\m@th\ialign{##\crcr
				\leftharpoonupfill \crcr
				\noalign{\kern-\p@\nointerlineskip}
				$\hfil\SavedStyle#1\hfil$\crcr}}}}

\def\leftharpoonupfill{%
	$\SavedStyle\m@th\mkern+0.8mu\cleaders\hbox{$\shortbar\mkern-4mu$}\hfill\leftharpoonuptip\mkern+0.8mu$}

\def\leftharpoonuptip{%
	\raisebox{\z@}[2pt][1pt]{\scalebox{0.55}{$\SavedStyle\leftharpoonup$}}}

\makeatletter
\newsavebox\myboxA
\newsavebox\myboxB
\newlength\mylenA

\newcommand*\xoverline[2][0.75]{%
	\sbox{\myboxA}{$\m@th#2$}%
	\setbox\myboxB\null
	\ht\myboxB=\ht\myboxA%
	\dp\myboxB=\dp\myboxA%
	\wd\myboxB=#1\wd\myboxA
	\sbox\myboxB{$\m@th\overline{\copy\myboxB}$}
	\setlength\mylenA{\the\wd\myboxA}
	\addtolength\mylenA{-\the\wd\myboxB}%
	\ifdim\wd\myboxB<\wd\myboxA%
	\rlap{\hskip 0.5\mylenA\usebox\myboxB}{\usebox\myboxA}%
	\else
	\hskip -0.5\mylenA\rlap{\usebox\myboxA}{\hskip 0.5\mylenA\usebox\myboxB}%
	\fi}
\makeatother

\begin{document}
	
	\title[Simplicial Tur\'an problems]
	{Simplicial Tur\'an problems}
	
	\author[D.~Conlon]{David Conlon}
	\address{Department of Mathematics, California Institute of Technology, USA}
	\email{dconlon@caltech.edu}
    \thanks{\emph{Data availability statement.} There are no additional data beyond that contained within the main manuscript.}
	
	\author[S.~Piga]{Sim\'on Piga}
	\address{School of Mathematics, University of Birmingham, UK}
	\email{s.piga@bham.ac.uk}
    \thanks{The research leading to these results was supported by NSF Award DMS-2054452 (D.~Conlon) and EPSRC Grant No. EP/V002279/1 (S.~Piga).
    }
	
	\author[B.~Sch\"ulke]{Bjarne Sch\"ulke}
	\address{Department of Mathematics, California Institute of Technology, USA}
	\email{schuelke@caltech.edu}

	
	\begin{abstract}
        A simplicial complex~$H$ consists of a pair of sets~$(V,E)$ where~$V$ is a set of vertices and~$E\subseteq\mathscr{P}(V)$ is a collection of subsets of $V$ closed under taking subsets.
        Given a simplicial complex~$F$ and~$n\in \mathds N$, the extremal number~$\ex(n,F)$ is the maximum number of edges that a simplicial complex on~$n$ vertices can have without containing a copy of~$F$.
        We initiate the systematic study of extremal numbers in this context by asymptotically determining the extremal numbers of several natural simplicial complexes.
        In particular, we asymptotically determine the extremal number of a simplicial complex for which the extremal example has more than one incomplete layer.
	\end{abstract}
	
	\maketitle
	
	\section{Introduction}

        Tur\'an's paper~\cite{T:41} in which he showed that complete, balanced,~$(r-1)$-partite graphs maximise the number of edges among all~$K_r$-free graphs is often seen as one of the starting points of extremal combinatorics.
    Already in this paper, he asked for similar results for hypergraphs.
    More than $80$ years later, still little is known in this regard.
    
    A hypergraph~$H=(V,E)$ consists of a vertex set~$V$ and an edge set~$E\subseteq \ccP(V)=\{e\subseteq V\}$.
    We say that~$H$ is~$k$-uniform (or a~$k$-graph) if~$E\subseteq V^{(k)}=\{e\subseteq V:\vert e\vert=k\}$.
    The extremal number for~$n\in\mathds{N}$ and a~$k$-graph~$F$ is the maximum number of edges in an~$F$-free~$k$-graph on~$n$ vertices and is denoted by~$\ex^{(k)}(n,F)$.
    The Tur\'an density of~$F$ is then $\pi(F)=\lim_{n\to\infty}{\ex^{(k)}(n,F)}{\binom{n}{k}}^{-1}$ (this limit was shown to exist by monotonicity in~\cite{KNS:64}).
    The problem of determining the extremal numbers or the Tur\'an densities of hypergraphs is often referred to as the Tur\'an problem.
    For graphs, this problem is reasonably well understood, as a result of Erd\H{o}s, Stone, and Simonovits~\cites{ES:46,ES:66}, generalising Tur\'an's theorem, says that the Tur\'an density of any graph~$F$ is~$\frac{\chi(F)-2}{\chi(F)-1}$, where~$\chi(F)$ is the chromatic number of~$F$.
    
    In contrast, for hypergraphs, only a few Tur\'an densities are known, one prominent example being that of the Fano plane~\cites{DF:00,FS:05,KS:05,BR:19}.
    Even the Tur\'an density of the complete~$3$-graph on four vertices,~$K_4^{(3)}$, is still not known, although it is widely conjectured to be~$5/9$.
    In fact, the problem is already open for the~$3$-graph that consists of three edges on four vertices, denoted by~$K_4^{(3)-}$.
    For more on the hypergraph Tur\'an problem, we refer the interested reader to the survey by Keevash~\cite{K:11}.
    
    In this work, we consider the Tur\'an problem for (abstract) simplicial complexes.
    A hypergraph~$H=(V,E)$ is a simplicial complex if for all~$e\in E$ and~$e'\subseteq e$, we have~$e'\in E$.
    In other words, a simplicial complex is a hypergraph whose edge set is closed under taking subsets (we also say that~$E$ is hereditary). The~$k$-uniform layer of a simplicial complex~$H=(V,E)$ is~$H_k=(V,E^{(k)})$, where~$E^{(k)}=\{e\in E:\vert e\vert=k\}$.
    In line with the geometric interpretation of simplicial complexes, we define the dimension of~$H$ to be $$\dim(H)=\max\{k\geq 0:E^{(k)}\neq\emptyset\}-1\,.$$
    For brevity, when describing the edge set of a simplicial complex, we will generally not mention the induced edges, i.e., the edges which are properly contained in another edge.
    
    The extremal number for~$n\in \mathds{N}$ and a simplicial complex~$F$ is the maximum number of edges in an~$F$-free simplicial complex on~$n$ vertices and is denoted by~$\ex(n,F)$.
    For a family of simplicial complexes~$\mathcal F$, we define~$\ex(n,\mathcal F)$ analogously. We are by no means the first to study the extremal problem for simplicial complexes. For example, already in 1983,
    Frankl~\cite{F:83} observed that $\ex(n,F)=\sum_{i=0}^{s-1}\binom{n}{i}$ for $F=\big([s],\ccP([s])\big)$ and derived the celebrated Sauer--Shelah--Vapnik--Chervonenkis lemma~\cite{S:72,Sh:72,VC:68} as a corollary. However, we believe that we are the first to study the problem systematically, asymptotically determining the extremal numbers of several natural simplicial complexes. 

    As a first attempt, one might hope for some way to reduce the Tur\'an problem for simplicial complexes to that for uniform hypergraphs.
    Indeed, given a simplicial complex~$F$ of dimension~$k-1$, let $H$ be the simplicial complex with~$|V(H)|=n$, $E^{(i)}=V(H)^{(i)}$ for~$0\leq i\leq k-1$ and $E^{(k)}$ the edge set of a~$k$-graph on~$V(H)$ that is $F_k$-free. Then $H$ is clearly $F$-free, yielding the following general lower bound.
    
    \medskip
    
    \begin{observation}\label{obs:lowerbound}
        If~$F$ is a simplicial complex of dimension~$k-1$, then
        \begin{align*}
            \ex(n,F)\geq \ex^{(k)}(n,F_k)+\sum_{i=0}^{k-1}\binom{n}{i}\,.
        \end{align*}
    \end{observation}

 \medskip

    We will show that this bound is asympotically optimal for several simplicial complexes. However, this is not always the case. To say more, we recall an old conjecture of Hajnal (originally formulated in terms of traces of sets rather than simplicial complexes). 
    Given positive integers $k$ and $d$ with~$k > d$, let~$\cF(k,d)$ be the family of all simplicial complexes with vertex set~$[k]$ and~$1+\sum_{i=0}^{d}\binom{k}{i}$ edges.
    The problem of determining~$\ex(n,\cF(k,d))$ can be seen as a variant of the well-known Brown--Erd\H{o}s--S\'os problem~\cite{SEB:73} for simplicial complexes.
    Hajnal's conjecture (see~\cite{B:72,FP:94}) is then that~$\ex(n,\cF(k,d))=\sum_{i=0}^{d}\binom{n}{i}$, in keeping with the lower bound in Observation~\ref{obs:lowerbound}. 
    Bollob\'as and Radcliffe~\cite{BR:95} proved this conjecture for~$k=4$ and~$d=2$.
    However, in full generality, it was disproved by Frankl~\cite{F:78} and later, in a very strong sense, by Bukh and Goaoc~\cite{BG:19}.
    In particular, from their result it follows that for every~$d\geq 2$ and every exponent~$\ell\in \mathds N$, there is a sufficiently large~$k$ such that~$\ex(n,\mathcal F(k,d)) = \Omega(n^\ell)$. 
    However, for~$k\geq 5$ and~$d\geq 2$, the problem of determining the asymptotics of~$\ex(n,\cF(k,d))$ remains wide open.
    Indeed, it seems to be significantly harder to determine the extremal number of a simplicial complex if it does not agree with the lower bound in Observation~\ref{obs:lowerbound}.
    
    Our main result, Theorem~\ref{thm:bowtie} below, breaks through this barrier, asymptotically determining the extremal number of a simplicial complex~$B$ for which the extremal number is not given by the bound in Observation~\ref{obs:lowerbound}.
    To the best of our knowledge, this is the first result of this type and shows that 
    the extremal behaviour of simplicial complexes can be genuinely different from that of the uniform hypergraph given by the largest layer.
	
	\begin{theorem}\label{thm:bowtie}
		For~$B=(\{v_1,\dots,v_5\},\{v_3v_4,v_1v_2v_3,v_1v_4v_5\})$, there is a constant~$C\in\mathds{N}$ such that 
		$$\ex(n,B)\leq\frac{9}{8}\binom{n}{2}+\frac{21}{16}n+C\,$$
		for every~$n\geq 5$.
	\end{theorem}

    As shown by the following example, this bound is best possible up to the constant.
    
    \medskip
    
	\begin{exmpl}\label{ex:extremalbowtie}
		Let~$n\in\mathds{N}$ with~$n\equiv 6 \mod{8}$, let~$V=V_1\dcup V_2$ for some sets~$V_1$ and~$V_2$ with~$\vert V_2\vert$ even and~$\vert V_1\vert+\vert V_2\vert=n$, and let~$\cM$ be a perfect~$2$-uniform matching on~$V_2$.
        Furthermore, let
        \begin{align*}
            E'=\{abc:ab\in\cM\text{ and }c\in V_1\}\cup V_1^{(2)}
        \end{align*}
        and let~$E$ be the downward closure of~$E'$.
        Then set~$H=(V,E)$ and note that~$B\not\subseteq H$ and
        \begin{align*}
            \vert E\vert=&\frac{n-\vert V_1\vert}{2}\vert V_1\vert+(n-\vert V_1\vert)\vert V_1\vert+\binom{\vert V_1\vert}{2}+\frac{n-\vert V_1\vert}{2}+n+1\,.
        \end{align*}
        The right-hand side is maximised if~$\vert V_1\vert=\frac{3}{4}n-\frac{1}{2}$, in which case it is equal to~$\frac{9}{8}\binom{n}{2}+\frac{21}{16}n+\frac{5}{4}$.
	\end{exmpl}

	\medskip

    The proof of Theorem~\ref{thm:bowtie} is the longest and most difficult part of this paper.
    We believe that the approach used for this proof might have further applications for hypergraph Tur\'an-type problems.
    Roughly speaking, it proceeds by gaining more and more structure on an extremal~$B$-free simplicial complex, ultimately yielding the same structure as Example~\ref{ex:extremalbowtie}.
    This is done by introducing and analysing a parameter~$\rho$, which is defined for every vertex and depends on its link, and symmetrising a subset of the vertices determined by~$\rho$.
    
    We now state our results where, up to lower order terms, the extremal number is given by the lower bound in Observation~\ref{obs:lowerbound}.
    We begin with the comparatively simple case of two disjoint $3$-edges connected by a $2$-edge.
    
	
	\begin{theorem}\label{thm:match}
	    For the simplicial complex~$F=(\{v_1,\dots,v_6\},\{v_1v_2v_3,v_4v_5v_6,v_1v_4\})$, there is~$C\in \mathds N$ such that~$n^2-n+2 \le \ex(n,F) \le n^2-n+C$.
	\end{theorem}
	
	Given a~$3$-graph~$F$ and~$A\subseteq V(F)$, let~$F^+$ be the simplicial complex obtained from~$F$ by adding a new vertex~$a$ and all~$2$-edges from~$a$ to~$A$ (and then taking the downward closure).
    Our next theorem follows from a more general result, Theorem~\ref{thm:2connvtxgeneral-2}, that we prove in Section~\ref{sec:gen}.
    It gives a general condition under which the extremal number of~$F^+$ is approximately equal to the bound in Observation~\ref{obs:lowerbound} provided we also require that the host simplicial complex is~$2$-dimensional.
    Let us call a~$3$-graph \emph{proper} if for every~$\varepsilon>0$, there is some~$n_0\in\mathds{N}$ such that for every~$n\geq n_0$, we have
    $$(1+\varepsilon)\left[\ex^{(3)}(n,F)-\ex^{(3)}(n-1,F)\right]\geq\frac{\ex^{(3)}(n,F)}{n}\,.$$
    For instance, if~$\ex^{(3)}(n,F)\approx Cn^\alpha$ for some~$C>0$ and~$\alpha\geq 1$, then $F$ is proper.
    Moreover, by duplicating a vertex that has at least the average degree in an extremal example for~$F$ on~$n-1$ vertices (combined with a standard double counting argument from~\cite{KNS:64}), one can easily see that if the shadow of~$F$ is complete, then~$F$ is proper.
    
    \begin{theorem}\label{thm:2connvtxgeneral}
        For every proper~$3$-graph~$F$,~$A\subseteq V(F)$ and~$\eta,\varepsilon>0$, there is~$n_0 \in \mathbb{N}$ such that the following holds.
	    If a~$2$-dimensional simplicial complex~$H$ on~$n\geq n_0$ vertices satisfies
	    \begin{align}\label{eq:condonmaxdeg1}
	        \frac{\ex^{(3)}(n,F)}{n^2}+1\geq\Big(\frac{(\vert A\vert-1)}{\vert A\vert}+\eta\Big)\big[1+\frac{\Delta^{(3)}}{2}\big]\,,
	   \end{align}
	   where~$\Delta^{(3)}=\max_{xy\in V(H)^{(2)}}d^{(3)}(xy)$,  and~$e(H)\geq(1+\varepsilon)\big[\ex^{(3)}(n,F)+\binom{n}{2}\big]$, then~$H$ contains a copy of~$F^+$.
	\end{theorem}
	
	The following corollary of Theorem~\ref{thm:2connvtxgeneral-2} (using that the extremal number of the~$3$-graph with two edges on four vertices has been determined in~\cite{FF:87}) asymptotically determines the extremal number of another natural simplicial complex on five vertices.
	
	\begin{cor}\label{cor:2connvtx}
	    For the simplicial complex~$F=(\{v_1,\dots,v_5\},\{v_1v_2v_3,v_2v_3v_4,v_2v_5,v_3v_5\})$ and~$\varepsilon>0$, there is~$n_0 \in \mathbb{N}$ such that~$\ex(n,F)=(\frac{4}{3}\pm\varepsilon)\binom{n}{2}$ for every~$n\geq n_0$.
	\end{cor}
	
    Finally, we determine the asymptotics for the extremal number of the simplicial complex on four vertices given by two~$3$-edges that intersect in two vertices and a~$2$-edge between the two vertices which do not lie in the intersection.
    
	\begin{theorem}\label{thm:Bollobas}
	For the simplicial complex~$F=(\{v_1,v_2,v_3,v_4\},\{v_1v_2v_3,v_2v_3v_4,v_1v_4\})$ and~$\varepsilon>0$, there is~$n_0 \in \mathbb{N}$ such that~$ex(n,F)=\big(\frac{2}{9}\pm\varepsilon\big)\binom{n}{3}$ for every~$n\geq n_0$.
	\end{theorem}

	\subsection*{Notation}\label{sec:prelim}
		For a simplicial complex~$H=(V,E)$,~$S\subseteq V$,  and~$i\in\mathds{N}$, we define the~$i$-uniform degree of~$S$ by~$d^{(i)}(S)=\vert\{e\in E:S\subseteq e,\vert e\vert=i\}\vert$.
		Omitting parentheses around~$1$- and~$2$-sets, we have, in particular, that~$d^{(i)}(v)=\vert\{e\in E:v\in e,\vert e\vert=i\}\vert$ for $v \in V$.
		We also write~$d^{(\geq i)}(v)=\sum_{j\geq i}d^{(j)}(v)$.
        The degree of~$v$ counts all edges containing~$v$ except the singleton, that is, $d(v)=\sum_{i\in\mathds{N},i\geq 2}d^{(i)}(v)$.
		The neighbourhood of~$v$ is~$N(v)=\{w\in V\setminus v:\exists e\in E:v,w\in e\}$ and, more generally, the neighbourhood of a set~$S$ is~$N(S)=\{w\in V\setminus S:\exists e\in E:S\subseteq e,w\in e\}$.
        When considering simplicial complexes and the context is clear, we will usually only mention those edges whose downward closure is the edge set.
        For a vertex~$v\in V$, let~$L_v$ be the link of~$v$ in~$H$, i.e., the simplicial complex with vertex set~$V-v$ and edge set~$\{e-v:v\in e\in E\}$.
        Finally, we write~$E^{(i)}(H)=\{e\in E\colon |e|=i\}$,~$e^{(i)}(H)=|E^{(i)}|$, and $e(H)=\vert E\vert$. 
        Similarly, we write~$v(H)=\vert V\vert$.

        \section{Proof of Theorem~\ref{thm:bowtie}}
        Let~$H=(V,E)$ be a simplicial complex with~$\vert V\vert=n$,~$B\not\subseteq H$, and assume for the sake of contradiction that~$\vert E\vert\geq \frac{9}{8}\binom{n}{2}+\frac{21}{16}n+C$, where~$C$ is some sufficiently large constant.
        
        Roughly speaking, we proceed as follows.
        By induction on the number of vertices, we may assume some minimum degree.
        The proof is then centered around a parameter~$\rho(v)$, defined for every vertex~$v$.
        Depending on the role that~$v$ plays in the link of another vertex~$w$, we know how~$\rho(v)$ and~$\rho(w)$ relate.
        This allows us to determine the rough structure of~$H$.
        By continuing to analyse~$\rho$ and by symmetrising a subset of the vertices determined by~$\rho$, we can simplify this structure further. 
        Finally, we obtain a~$B$-free simplicial complex which equals the extremal example aside from maybe the sizes of the two partition classes and has at least as many edges as the initial~$H$.
        A simple optimisation then finishes the proof.
        
        Let us formalise the proof by performing an induction on~$n$.
        Let~$n_0\in \mathds N$ be sufficiently large and choose~$C\in \mathds N$ large enough that every simplicial complex on~$n<n_0$ vertices with at least~$\frac{9}{8}\binom{n}{2}+\frac{21}{16}n + C$ edges contains a copy of~$B$. 
        We take those cases as the base of our induction.
        The remainder of this section is devoted to the induction step.

        For~$n\geq n_0$, assume that every simplicial complex on~$n-1$ vertices with at least~$\frac{9}{8}\binom{n-1}{2}+\frac{21}{16}(n-1)+C$ edges contains~$B$.
        Let~$H=(V,E)$ be a simplicial complex with~$\vert V\vert=n$,~$B\not\subseteq H$, and assume for the sake of  contradiction that~$\vert E\vert\geq \frac{9}{8}\binom{n}{2}+\frac{21}{16}n+C$.
        If there is a vertex~$v\in V$ with~$d(v)\leq \frac{9}{8}(n-1)+\frac{5}{16}$, we have that~$\vert E(H-v)\vert\geq \frac{9}{8}\binom{n}{2}+\frac{21}{16}n+C-d(v)-1\geq\frac{9}{8}\binom{n-1}{2}+\frac{21}{16}(n-1)+C$ and so, by induction,~$H-v$ contains a copy of~$B$.
        Therefore, we may assume that for every~$v\in V$,
        \begin{align}\label{eq:mindeg}
            d(v)>\frac{9}{8}(n-1)+\frac{5}{16}\,.
        \end{align}

        Clearly,~$H$ cannot contain any edges of size at least five.
        Thus, the links consist only of~$3$-edges,~$2$-edges, and~$1$-edges (and the empty set).
        Furthermore, we have the following.
        
        \begin{claim}\label{cl:no3path}
            For every~$v\in V$, there is no~$2$-uniform path with three edges in~$L_v$.
        \end{claim}
        
        \begin{proof}
            Assume that there is a~$2$-uniform path~$abcd$ in the link of some~$v\in V$.
            Since~$H$ is a simplicial complex, we have~$bc\in E$ and, by the definition of the link, we have~$abv,cdv\in E$.
            Thus, there is a copy of~$B$ in~$H[a,b,c,d,v]$, a contradiction.
        \end{proof}

        For a simplicial complex~$F$ and a subcomplex~$F'\subseteq F$, we say that~$F'$ is \emph{isolated} in~$F$ if~$F$ does not contain an edge containing both a vertex of~$F'$ and a vertex in~$V(F)\setminus V(F')$.
        Claim~\ref{cl:no3path} means that every link consists only of: isolated vertices, isolated ($2$-)edges, isolated induced ($2$-uniform) stars with at least two edges, isolated ($2$-uniform) triangles, and potentially isolated~$3$-edges.
        However, the next claim states that $H$ does not have any $4$-edges, so the links do not in fact contain any $3$-edges.
        \begin{claim}\label{cl:nolargeedges}
            There are no~$4$-edges in~$H$.
        \end{claim}
        \begin{proof}[Proof]
            Assume that~$e=v_1v_2v_3v_4$ is a~$4$-edge.
            Note that no~$3$-edge can intersect~$e$ in exactly two vertices, as otherwise~$B\subseteq H$.
            For~$i\in[4]$, let~$N_i$ be the set of vertices in~$V\setminus e$ which are contained in a~$3$-edge with~$v_i$.
            Note that for all distinct~$i,j\in[4]$, we have that~$v_i$ has no~$2$-edge to any vertex in~$N_j$ (as otherwise, again,~$B\subseteq H$), so in particular~$N_i$ and~$N_j$ are disjoint.
            Without loss of generality, we may assume that~$\vert N_1\vert\leq\dots\leq\vert N_4\vert$.
            Since all $2$-edges between~$v_1$ and the vertices in~$N_2\dcup N_3\dcup N_4$ are missing, $d^{(2)}(v_1) \leq (n-1) - (|N_2| + |N_3| + |N_4|)$. Together with~\eqref{eq:mindeg}, this implies that  
            \begin{align*}
            d^{(3)}(v_1)+d^{(4)}(v_1) > \vert N_2\vert+\vert N_3\vert+\vert N_4\vert+\frac{n-1}{8}+\frac{5}{16}. 
            \end{align*}
            On the other hand, by Claim~\ref{cl:no3path}, $L_{v_1}[N_1]$ does not contain any~$2$-uniform path on three edges.
            It is easy to check that this means the number of~$2$- and~$3$-edges in~$L_{v_1}$ is at most~$4|N_1|/3$. 
            Adding the $3$- and~$4$-edges in~$e$ containing~$v_1$, we obtain 
            \begin{align*}
            d^{(3)}(v_1)+d^{(4)}(v_1) \leq 4+\frac{4}{3}\vert N_1\vert, 
            \end{align*}
            which is a contradiction for~$n$ sufficiently large.
        \end{proof}
			
		\subsection*{Definition and basic properties of~$\rho$.}
            Let us now introduce the central parameter of the proof.
			Let~$\varphi(v)$ denote the number of isolated~$2$-edges in~$L_v$ and define
			\begin{align*}
				\rho(v)=e^{(2)}(L_v)+\varphi(v)-2\,.
			\end{align*}
			
			The idea behind this parameter is that it is (a lower bound on) the number of~$2$-edges that must be missing at a vertex~$w$ if~$w$ appears in an isolated~$2$-edge in~$L_v$.
			We make this precise with the following claim.
			
			\begin{claim}\label{cl:rhoup}
				Let~$v,w,x\in V$ be such that~$wx$ is an isolated edge in~$L_v$.
				Then $d^{(2)}(w)\leq n-1-\rho(v)$.
                Moreover, 
                $$\rho(w)\geq d^{(3)}(w)-2> \frac{n-1}{8}-\frac{27}{16}+\rho(v)\,.$$
			\end{claim}
			
			\begin{proof}
                Let~$ab\in E^{(2)}(L_v)\setminus \{wx\}$ and note that since~$wx$ is an isolated edge in~$L_v$, we have that~$a,b,w,x$ are distinct.
                If~$aw\in E$, then~$H[a,b,v,w,x]$ contains a copy of~$B$, a contradiction.
                Thus, the only~$2$-edge in~$H$ formed by~$w$ and a vertex that is contained in a~$2$-edge of~$L_v$ is~$wx$.
                In other words,~$w$ is `missing' a~$2$-edge to every vertex spanned by a~$2$-edge in~$E^{(2)}(L_v)\setminus \{wx\}$.

                Recall that, by Claims~\ref{cl:no3path} and~\ref{cl:nolargeedges},~$L_v$ consists only of isolated vertices, isolated $2$-edges, isolated induced ($2$-uniform) stars with at least two edges, and isolated ($2$-uniform) triangles.
                Clearly, every isolated~$2$-edge in~$L_v$ spans exactly two vertices. 
                Moreover, we have that isolated ($2$-uniform) triangles and stars span at least as many vertices as they have~$2$-edges. 
                Therefore, the total number of vertices spanned by $E^{(2)}\setminus \{wx\}$ is at least~$e^{(2)}(L_v) + \varphi(v) - 2 = \rho(v)$.
                By the argument in the paragraph above, each of those vertices corresponds to a missing~$2$-edge in~$w$. 
                This yields the first inequality of the claim.
				
				The moreover part simply follows from the first part,~\eqref{eq:mindeg}, and the definition of~$\rho$.
			\end{proof}
		
			The following claim is in a sense the reverse of the previous one.
			For~$v,w\in V$, we say that~$w$ is a \emph{tip of a star in~$L_v$}, and write~$w\leftarrow v$, if there are~$x,y\in V$ such that~$wx,xy\in E(L_v)$.
			
			\begin{claim}\label{cl:rhodown}
				Let~$v,w\in V$ be such that~$w\leftarrow v$.
				Then~$d^{(2)}(v)\leq n-\rho(w)$.
				Moreover, $$\rho(v)\geq d^{(3)}(v)-2> \frac{n-1}{8}-\frac{43}{16}+\rho(w)\,.$$
			\end{claim}
		
			\begin{proof}
				Let~$x,y\in V$ be such that~$wx,xy\in E(L_v)$ and note that~$vx\in E(L_w)$.
				If~$vx$ is an isolated edge in~$L_w$, the statement follows from Claim~\ref{cl:rhoup}.
				If there is some~$y'\in V\setminus \{v,x,y\}$ such that~$vy'\in E(L_w)$ or~$xy'\in E(L_w)$, then~$H[v,w,x,y,y']$ contains~$B$.
				Thus, we are left with the remaining case that both~$v$ and~$x$ lie in at most two~$2$-edges in~$L_w$, namely,~$vx$ and possibly~$vy$ and~$xy$.
				As in the proof of Claim~\ref{cl:rhoup}, if there is~$ab\in E^{(2)}(L_w)$ with~$ab\cap\{v,x,y\}=\emptyset$ and~$av\in E$, then~$B\subseteq H[a,b,v,w,x]$.
                That is,~$v$ is missing a~$2$-edge to every vertex spanned by an edge in~$E^{(2)}(L_w)\setminus\{vx, vy, xy\}$.
                Again, as in the proof of Claim~\ref{cl:rhoup}, there are at least~$e^{(2)}(L_w)-3+\varphi(w)=\rho(w)-1$ such vertices. 
				This yields the first part of the claim.
				
				The moreover part again follows from the first part,~\eqref{eq:mindeg}, and the definition of~$\rho$.
			\end{proof}
			
			At this point, we already begin to see some resemblance of~$H$ to the extremal example, namely, in the structure of the links of those vertices with extremal~$\rho$.
			
			\begin{cor}\label{cor:minmaxrho}
				Let~$v\in V$.
				\begin{enumerate}
					\item\label{it:minrho} If~$\rho(v)=\min_{w\in V}\rho(w)$, then the~$2$-edges of~$L_v$ form a matching of size larger than~$\frac{n-1}{8}+\frac{5}{16}$.
					\item\label{it:maxrho} If~$\rho(v)=\max_{w\in V}\rho(w)$, then the~$2$-edges of~$L_v$ do not contain any isolated edges.
                    \item\label{it:notriangle} The link of~$v$ consists only of isolated vertices, isolated $2$-edges, and isolated induced ($2$-uniform) stars with at least two edges.
				\end{enumerate}
			\end{cor}
		
			\begin{proof}
				For~\eqref{it:minrho}, note that if the~$2$-edges of~$L_v$ do not form a matching, then there would be a vertex~$w$ that is a tip of a star in~$L_v$, implying that~$\rho(w)<\rho(v)$ by  the second inequality in Claim~\ref{cl:rhodown}.
                Moreover,~\eqref{eq:mindeg} implies that~$d^{(3)}(v)> \frac{n-1}{8}+\frac{5}{16}$.
				
				For~\eqref{it:maxrho}, note that if there is an isolated~$2$-edge~$xy$ in~$L_v$, then~$\rho(x)>\rho(v)$ by Claim~\ref{cl:rhoup}.

                To see~\eqref{it:notriangle}, recall that we already know that~$L_v$ consists only of isolated vertices, isolated $2$-edges, isolated induced ($2$-uniform) stars with at least two edges, and isolated ($2$-uniform) triangles.
                Assume now that there is a triangle with vertices~$w$,~$x$, and~$y$ in~$L_v$.
                Then considering the link of~$w$ yields~$x\leftarrow w$ (since~$x$ and~$y$ are tips of a star centered in~$v$ in~$L_w$) and thus~$\rho(x)<\rho(w)$ by  Claim~\ref{cl:rhodown}.
                However, considering the link of~$x$ similarly gives~$w\leftarrow x$ and so~$\rho(w)<\rho(x)$, a contradiction.
			\end{proof}
		
			Next, we bring the structure closer to the extremal example by performing some symmetrisations.
			
		\subsection*{Symmetrisation of vertices}
			
			We begin by defining~$V_1$ as the set of vertices in~$H$ for which the~$2$-edges in the link form a matching.
			Note that, by Corollary~\ref{cor:minmaxrho} \eqref{it:minrho}, we have~$V_1\neq \emptyset$. 
            Moreover, no two vertices of~$V_1$ are contained in the same~$3$-edge.
	        Indeed, if~$v,w \in V_1$ are such that~$wx\in L_v$ for some vertex~$x\in V$, then, by Claim~\ref{cl:rhoup}, we have~$\rho(w)>\rho(v)$.
            But since~$vx\in L_w$ as well, we have~$\rho(w)<\rho(v)$, a contradiction.
            
			Now pick~$v\in V_1$ with~$d(v)$ maximal among the vertices in~$V_1$.
			Clearly, if we replace each vertex in~$V_1$ (one by one) by a copy of~$v$ and add all~$2$-edges between these vertices, then the total number of edges does not decrease. 
            Moreover, recalling that the~$2$-edges of $L_v$ form a matching, it is not hard to check that no new copies of~$B$ are created in this way.
            Thus, from now on we can consider $H$ to be the resulting simplicial complex.
            In particular, all vertices~$v\in V_1$ have the same link~$L_v$ whose~$2$-edges form a fixed matching~$M$.
            
            Note that, by~\eqref{eq:mindeg}, we have 
            $\vert M\vert> \frac{n-1}{8}+\frac{5}{16}$ and therefore, by the definition of~$\rho(\cdot)$, we have, for all~$v\in V_1$, that 
            \begin{align}\label{eq:minrho}
                \rho(v)> 2\left(\frac{n-1}{8}+\frac{5}{16}\right)-2 = \frac{n-1}{4} - \frac{11}{8}\,.
            \end{align}
			
			The following notation will help us as we make more use of Claims~\ref{cl:rhoup} and~\ref{cl:rhodown}.
			For~$v,w\in V$ and~$i\in\mathds{N}$, write~$v\overset{i}{\leftarrow} w$ if there are vertices~$v=v_0,\dots,v_i=w$ such that, for all~$j\in[i]$, we have that~$v_{j-1}\leftarrow v_j$.
			Furthermore, for~$S\subseteq V$, write~$S\overset{i}{\leftarrow} w$ if there is some~$v\in S$ such that~$v\overset{i}{\leftarrow} w$.
			Note that~$v\overset{1}{\leftarrow} w$ just means~$v\leftarrow w$.
			
			\begin{claim}\label{cl:nolongbacktrack}
				There is no~$v\in V$ and~$i\geq 3$ such that~$V_1\overset{i}{\leftarrow} v$.
			\end{claim}
			
			\begin{proof}
				Assume for the sake of contradiction  that~$v\in V$ and~$i\in\mathds{N}$ with~$i\geq 3$ such that~$V_1\overset{i}{\leftarrow}v$.
				Then there are~$v_0\in V_1$ and~$v_1,v_2,v_3\in V$ (with~$v_3$ being~$v$ if~$i=3$) such that~$v_0\leftarrow v_1\leftarrow v_2\leftarrow v_3$.
                By~\eqref{eq:minrho}, we know that~$\rho(v_0)> \frac{n-1}{4}-\frac{11}{8}$.
				Applying Claim~\ref{cl:rhodown} twice yields~$\rho(v_2)> \frac{n-1}{4}-\frac{11}{8}+2\big(\frac{n-1}{8} - \frac{43}{16}\big)= \frac{n-1}{2}-\frac{27}{4}$.
				This means, again by Claim~\ref{cl:rhodown}, that~$d^{(2)}(v_3)\leq n - \rho(v_2)\geq n - \frac{n-1}{2} + \frac{27}{4} = \frac{n-1}{2}+\frac{31}{4}$ and 
                $$d^{(3)}(v_3)
                >
                \frac{n-1}{8}-\frac{43}{16}+2+\rho(v_2)
                \geq
                \frac{5(n-1)}{8}-\frac{119}{16}\,.$$
                It follows that~$d^{(3)}(v_3)>d^{(2)}(v_3)$ for $n$ sufficiently large.
				However, Corollary~\ref{cor:minmaxrho}~\eqref{it:notriangle} implies that~$d^{(3)}(v_3)\leq d^{(2)}(v_3)-1$, yielding the required contradiction.
			\end{proof}
		
			Next, guided by the intuition coming from the extremal example,  we show that the vertices in~$V(M)$ (which we might also call~$V_2$) are exactly the vertices~$v$ for which~$V_1\overset{1}{\leftarrow} v$.
			
			\begin{lemma}\label{lem:V(M)assshould}
				For~$v\in V$, the following are equivalent:
				\begin{itemize}
					\item $v\in V(M)$.
					\item $V_1\overset{1}{\leftarrow}v$ holds and $V_1\overset{i}{\leftarrow}v$ does not hold for~$i\geq 2$.
				\end{itemize}
			\end{lemma}
			
			\begin{proof}
				``$\Leftarrow$'': Let~$x\in V_1$ be a tip of a star in~$L_v$.
				This means in particular that~$v$ is contained in a~$2$-edge in~$L_x$.
				But any such edge is in~$M$, so that~$v\in V(M)$.
				
				``$\Rightarrow$'': Let~$w\in V$ be such that~$vw\in M$ and let~$x\in V_1$.
				In particular,~$vw$ is an isolated edge in~$L_x$, so that~$\rho(v)> \rho(x)+\frac{n-1}{8}-\frac{27}{16}$ by Claim~\ref{cl:rhoup}.
				Note that~$wx\in E(L_v)$.
				If~$wx$ is an isolated edge in~$L_v$, then another application of Claim~\ref{cl:rhoup} would yield~$\rho(x)>\rho(v)+\frac{n-1}{8}-\frac{27}{16}$, a contradiction.
				Therefore,~$wx$ is not an isolated edge in~$L_v$, meaning that there is another $2$-edge intersecting $wx$.
				Since~$vw$ is an isolated edge in~$L_x$, such an edge cannot contain~$x$, so it must contain~$w$, implying that~$d^{(2)}_{L_v}(w)\geq 2$. 
                In other words,~$x$ is a tip of a star in~$L_v$, i.e.,~$x\overset{1}{\leftarrow} v$.
				
				By Claim~\ref{cl:nolongbacktrack}, it is enough to show that~$V\overset{i}{\leftarrow}v$ does not hold for any~$i\geq 2$.
				Assume for the sake of contradiction that it does hold, i.e., that there is some~$x \in V_1$ and~$r\in V$ such that~$x\leftarrow r\leftarrow v$.
    
                Note that $V_1\overset{1}{\leftarrow}r$ holds and $V_1\overset{i}{\leftarrow}r$ does not hold for~$i\geq 2$, since otherwise we would have~$V_1\overset{i+1}{\leftarrow}v$ for some~$i\geq 2$, contradicting Claim~\ref{cl:nolongbacktrack}.
                Hence, by ``$\Leftarrow$'', we know that~$r\in V(M)$.
                Observe that any~$2$-edge in~$H$ between vertices in distinct edges of~$M$ would yield a copy of~$B$ and that~$r\leftarrow v$ in particular implies that~$rv\in E(H)$.
                Thus,~$rv\in M$ (and so~$r=w$),
                implying that~$rx\in E(L_v)$.
                Finally, recall that~$r$ is a tip of a star in~$L_v$ and note that~$x$ cannot be a vertex of this star, since~$rv$ is an isolated~$2$-edge in~$L_x$.
                Therefore, in~$L_v$, the~$2$-edge $xr$ is a pendant edge to this star.
                However, this gives a ($2$-uniform) path with three edges in~$L_v$, contradicting Claim~\ref{cl:no3path}.
			\end{proof}
		
		\subsection*{Gaining further structure}
		
            Recall that Corollary~\ref{cor:minmaxrho}~\eqref{it:notriangle} states that the link of any vertex consists only of isolated vertices, isolated $2$-edges, and isolated ($2$-uniform) induced stars with at least two edges.
            The following key lemma says that vertices of~$V(M)$ contain exactly one star in their links. 
			
			\begin{lemma}\label{lem:V2vtsonly1star}
                If~$vw\in M$, then~$L_v$ contains exactly one isolated ($2$-uniform) induced star with at least two edges and  
                the centre of this star is~$w$. 
			\end{lemma}
		
			We first prove the following statement.
			
			\begin{claim}\label{cl:posstowlkdwn}
                If~$v\in V$ and~$w\notin V_1$ with~$w\leftarrow v$, then~$V_1\overset{i}{\leftarrow}v$ for some~$i\geq 2$.
			\end{claim}
		
			\begin{proof}
                Assume that the statement is not true and let~$v$ and~$w$ be a counterexample where~$\rho(w)$ is minimal.
				As~$w\notin V_1$, there is some vertex with~$2$-degree at least~$2$ in~$L_w$ and, therefore, some~$w'\in V$ with~$w'\leftarrow w$.
				Note that by Claim~\ref{cl:rhodown}, we have~$\rho(w')<\rho(w)$.
				If~$w'\in V_1$, then~$V_1\overset{2}{\leftarrow}v$, contrary to our assumption that~$v$ and~$w$ are a counterexample to the statement.
				Thus,~$w'\notin V_1$ and since~$\rho(w)$ was chosen to be minimal among all counterexamples, we have that~$w$ and~$w'$ are not a counterexample, meaning that there is some~$i\geq 2$ such that~$V_1\overset{i}{\leftarrow}w$.
				But then~$V_1\overset{i+1}{\leftarrow}v$, contrary to our assumption.
			\end{proof}
		
			Let us now prove Lemma~\ref{lem:V2vtsonly1star}.
   
			\begin{proof}[Proof of Lemma~\ref{lem:V2vtsonly1star}]
				First, note that a vertex in~$V(M)$ must have a ($2$-uniform) star with at least two edges in its link for otherwise it would be in~$V_1$.
				To show that~$v$ has at most one ($2$-uniform) star with at least two edges in its link, assume the opposite.
				Then there are~$x_i,y_i\in V$ for~$i\in[3]$ such that~$x_1\neq y_1$ and~$x_1x_2,x_1x_3,y_1y_2,y_1y_3\in E(L_v)$.

                First, let us observe that~$x_2,x_3,y_2,y_3\in V_1$.
                Indeed, assume that one of~$x_2,x_3,y_2,y_3$ is not contained in~$V_1$, say, without loss of generality,~$x_2$.
                By Claim~\ref{cl:posstowlkdwn} and the fact that~$x_2\leftarrow v$, we would then have~$V_1\overset{i}{\leftarrow} v$ for some~$i\geq 2$.
				But, by Lemma~\ref{lem:V(M)assshould}, this contradicts~$v\in V(M)$.
                Therefore,~$x_2,x_3,y_2,y_3\in V_1$.
				But then~$vx_1\in E(L_{x_2})$ and~$vy_1\in E(L_{y_2})$ contradict the facts that~$x_2$ and~$y_2$ have the same link and~$v$ lies in a matching edge in this link.
				
				Lastly, let us show that~$w$ is the centre of a star with at least two edges.
				Recall that the vertices in~$V_1$ are symmetrised, so that, in particular, every vertex in~$V_1$ forms an edge with~$w$ in~$L_v$.
				Hence, it is enough to show that~$\vert V_1\vert\geq 2$.
				For this consider a vertex~$a\in V(M)$ and 
                recall that 
                there must be a~$2$-uniform star with at least two edges in~$L_a$.
				If one of the leaves in this star is not in~$V_1$, then Claim~\ref{cl:posstowlkdwn} implies that~$V_1\overset{i}{\leftarrow}a$ for some~$i\geq 2$.
                But, by Lemma~\ref{lem:V(M)assshould}, this would  contradict~$a\in V(M)$.
                Thus, we indeed have~$\vert V_1\vert\geq 2$ and~$w$ is the centre of a star with at least two edges.
			\end{proof}

            \begin{remark}
                In the remainder of this subsection, we will show how to symmetrise the pairs that form edges in $M$. This is not strictly necessary for the proof, so the reader may already jump to the next subsection. However, we believe that this technique of iterative partial symmetrisation, where we symmetrise certain vertex subsets and certain subsets of the pairs of vertices at appropriate junctures, has the potential to be more broadly useful in the study of hypergraph Tur\'an problems, so we decided to include it regardless.
            \end{remark}
    
			We symmetrise the pairs that form edges in~$M$ by essentially replacing any pair by a pair of larger degree as long as there still exist pairs which do not `look the same'.
			More precisely, for distinct~$x,y\in V$, we set
            \begin{align*}
                E_x(x,y)=&\{e\setminus x:e\in E(H),e\cap\{x,y\}=x\}\,, \\
                E_y(x,y)=&\{e\setminus y:e\in E(H),e\cap\{x,y\}=y\}\,,\text{ and }\\
                E_{xy}(x,y)=&\{e\setminus \{x,y\}:e\in E(H),e\cap\{x,y\}=\{x,y\}\}.
            \end{align*}
            When the pair~$x,y$ is clear from  context we will omit the~$`(x,y)$' from the notation.
            If there are~$xy,x'y'\in M$ such that~$E_x\neq E_{x'}$, $E_y\neq E_{y'}$, or~$E_{xy}\neq E_{x'y'}$, we perform the following operation on~$H$.
            Without loss of generality, assume that
            \begin{align}\label{eq:biggerpair}
                \vert E_x\vert+\vert E_y\vert+\vert E_{xy}\vert\geq\vert E_{x'}\vert+\vert E_{y'}\vert+\vert E_{x'y'}\vert.
            \end{align}
            Delete the vertices~$x'$ and~$y'$ and all edges containing at least one of them (these are precisely $E_{x'}$, $E_{y'}$, and~$E_{x'y'}$).
            Then add vertices~$u$ and~$v$ as well as the sets of edges
            \begin{align*}
                E_u^+=\{e\cup u:e\in E_x\}\,, \quad
                E_v^+=\{e\cup v:e\in E_y\}\,, \quad\text{and} \quad
                E_{uv}^+=\{e\cup \{u,v\}:e\in E_{xy}\}
            \end{align*}
            and call the resulting simplicial complex~$H'$.
            Intuitively, this means that we replace the pair~$x'y'$ by a copy of the pair~$xy$. 
            We will say that~$u$ is a copy of~$x$ and that $v$ is a copy of~$y$.\footnote{Observe, however, that~$u$ is~\emph{not} a copy of~$x$ in the usual sense. 
            For instance, there are no edges containing both~$u$ and~$y$.}
            To see that this operation is well defined, note that no edge in~$E_x$, $E_y$, or~$E_{xy}$ contains a vertex from the pair~$x'y'$ and, therefore, we never try to add new edges containing the deleted vertices $x'$ and $y'$. 
			Indeed, if, for example, $x'$ is contained in an edge of $E_x$, then $x$ and $x'$ are in a common edge of $H$, so, since $H$ is hereditary,~$xx'\in E(H)$. But then, since~$xy,x'y'\in M$, 
            any vertex in~$V_1$ together with~$x,x',y,y'$ forms a copy of~$B$, a contradiction.
            Note that, because of~\eqref{eq:biggerpair}, we do not decrease the number of edges, i.e., we have~$\vert E(H')\vert\geq \vert E(H)\vert$ (and~$\vert V(H')\vert=\vert V(H)\vert=n$).
			
            We now show that this symmetrisation does not introduce a copy of~$B$, i.e.,~$B\not\subseteq H'$.
			Indeed, assume that~$B\subseteq H'$ on some vertices~$v_1,\dots,v_5$ with~$v_1v_2v_3,v_1v_4v_5,v_3v_4\in E(H')$.
			Since~$B\notin H$, we have~$\{x,y\}\cap\{v_1,\dots,v_5\}\neq \emptyset$ and~$\{u,v\}\cap\{v_1,\dots,v_5\}\neq\emptyset$.
   
			First assume that~$x,u\in\{v_1,\dots,v_5\}$ or~$y,v\in\{v_1,\dots,v_5\}$, say, without loss of generality, $x,u\in\{v_1,\dots,v_5\}$.
			Because there is no edge in~$H'$ that contains both a vertex from~$xy$ and one from~$uv$, we may assume  that~$v_2=x$ and~$u\in\{v_4,v_5\}$.
			For the same reason,~$v_1\notin\{x,y,u,v\}$.
            We will now argue that in~$H$ the vertex~$v_1$ is the centre of a~$2$-uniform star with at least two edges in~$L_x$.
            But, since~$v_1\neq y$, this will contradict  Lemma~\ref{lem:V2vtsonly1star}. 
            If~$\{u,v\}=\{v_4,v_5\}$, then~$v_3v_4\in E(H')$ implies that~$v_3\neq y$.
            Recalling that~$uv$ is a copy of~$xy$, we see that~$v_1xy\in E(H)$, whereby~$v_1\neq y$ is the centre of a~$2$-uniform star with at least two edges in~$L_x$ (namely,~$v_1v_3$ and~$v_1y$).
			This yields the aforementioned contradiction and so we infer that~$v\notin\{v_4,v_5\}$.
			But then, again using that~$u$ is a copy of~$x$, we see that~$v_1\neq y$ is still the centre of a~$2$-uniform star with at least two edges in~$L_x$ (namely,~$v_1v_3$ and~$v_1z$, where~$z\in\{v_4,v_5\}\setminus u$).
            Once again, this yields a contradiction.
			
			The remaining cases are where~$x,v\in\{v_1,\dots,v_5\}$ and~$y,u\notin\{v_1,\dots,v_5\}$ or~$y,u\in\{v_1,\dots,v_5\}$ and~$x,v\notin\{v_1,\dots,v_5\}$. Without loss of generality, we may assume that~$x,v\in\{v_1,\dots,v_5\}$ and~$y,u\notin\{v_1,\dots,v_5\}$.
			As before, we may also assume that~$v_2=x$ and~$v\in\{v_4,v_5\}$ and 
		we still know             
            that~$v_1\notin\{x,y,u,v\}$.
			From~$u\notin\{v_1,\dots,v_5\}$, it follows that for~$z\in\{v_4,v_5\}\setminus v$, we have~$z\neq u$ and so, since~$v$ is a copy of~$y$ and~$v_1vz\in E(H')$, we have that~$v_1yz\in E(H)$.
			But then~$y\notin\{v_1,\dots,v_5\}$ and~$xy\in E(H)$ yield~$B\subseteq H[v_1,x,v_3,y,z]$, a contradiction.

            Thus,~$H'$ is a~$B$-free simplicial complex on~$n$ vertices with at least as many edges as~$H$.
            We repeat this operation iteratively so long as there are~$xy,x'y'\in M$ such that~$E_x\neq E_{x'}$, $E_y\neq E_{y'}$, or~$E_{xy}\neq E_{x'y'}$.
            We may therefore indeed assume that in $H$ the pairs of $M$ are 
            `symmetrised', i.e., for all~$xy,x'y'\in M$, we have~$E_x=E_{x'}$,~$E_y=E_{y'}$, and~$E_{xy}=E_{x'y'}$.
			
		\subsection*{The vertices outside~$V_1\cup V(M)$}
            Observe that~$H[V_1\cup V(M)]$ already has the same structure as Example~\ref{ex:extremalbowtie}, except perhaps for the orders of~$V_1$ and $V(M)$.
			In this subsection, we conclude the proof by showing that~$V_3:=V\setminus (V_1\cup V(M))=\emptyset$.
			First note that, for every vertex~$v\in V_3$, we have
			\begin{align}\label{eq:bcktrckV3}
				V_1\overset{2}{\leftarrow} v\,.
			\end{align}
			Indeed, since~$v\notin V_1$, there is a tip of a star in~$L_v$, so  Claim~\ref{cl:posstowlkdwn} gives~$V_1\overset{i}{\leftarrow}v$ for some~$i\geq 1$.
			Since~$v\notin V(M)$,~$v$ is not contained in a~$3$-edge together with any vertex in~$V_1$ and so~$i\geq 2$.
            Lastly, Claim~\ref{cl:nolongbacktrack} yields~$i=2$.
			
			Now we show that there are no~$3$-edges with all vertices in~$V_3$.
			
			\begin{claim}\label{cl:no3edginV3}
				$E\cap V_3^{(3)}=\emptyset$.
			\end{claim}
			
			\begin{proof}
				Assume that~$abc\in E$ with~$a,b,c\in V_3$.
                Clearly, $ab\in L_c$ and, by Corollary~\ref{cor:minmaxrho}~\eqref{it:notriangle}, in~$L_c$ the pair~$ab$ is either an isolated $2$-edge or part of an isolated induced~($2$-uniform) star with at least two edges.
                A similar conclusion holds for the two other pairs in~$\{a,b,c\}$. 
                Suppose that two pairs of vertices, say $ab$ and~$ac$, appear as isolated~$2$-edges in $L_c$ and~$L_b$, respectively.
                Then, by  Claim~\ref{cl:rhoup}, this would imply that~$\rho(c)<\rho(b)$ and~$\rho(b)<\rho(c)$, respectively, which is a contradiction.
				Therefore, at least one of the pairs in~$\{a,b,c\}$ appears as part of an isolated induced ($2$-uniform) star with at least two edges in the relevant link. 
                But this means that at least one of the vertices~$a,b,c$, say~$a$, appears as a tip of a star in the link of one of the other two, say~$b$. 
                That is, $a\leftarrow b$. 
                But then $V_1\overset{2}{\leftarrow} a\leftarrow b$, so that~$V_1\overset{3}{\leftarrow}b$,  contradicting Claim~\ref{cl:nolongbacktrack}.
			\end{proof}
		
			Next, we show that if~$V_3$ is non-empty, then~$V(M)$ must be very large.
			\begin{claim}\label{cl:V2largeifV3noem}
				If~$V_3\neq\emptyset$, then~$\vert V(M)\vert>n-\frac{21}{2}$.
			\end{claim}
			
			\begin{proof}
				Let~$v\in V_3$.
				If there is a~$3$-edge~$uvw\in E$ with~$u\in V_1$, then~$vw$ is an edge in~$L_u$ and since~$L_u$ only contains isolated~$2$-edges, this would give~$v\in V(M)$, a contradiction.
				Together with Claim~\ref{cl:no3edginV3}, this implies that every~$3$-edge containing~$v$ contains a vertex of~$V(M)$.
				On the other hand, Lemma~\ref{lem:V2vtsonly1star} states that for every~$xy\in M$,~$L_x$ contains exactly one ($2$-uniform) star with at least two edges and that the centre of this star is~$y$.
				The leaves of this star are all in~$V_1$ by Claim~\ref{cl:posstowlkdwn} and Lemma~\ref{lem:V(M)assshould}.
				Since~$v\notin V_1$, this implies that if~$v$ lies in a~$3$-edge with~$x$, then it lies in exactly one such $3$-edge (which appears as an isolated~$2$-edge in~$L_x)$.
                Note also that any~$3$-edge in~$H$ intersects at most one~$xy\in M$, since~$H$ does not contain any~$2$-edge between two vertices in distinct~$xy,x'y'\in M$.
				
				Recall that~$vxy$ cannot be an edge (because then,  in~$L_x$,~$v$ would be a leaf of the star centred in~$y$).
				Hence, assuming that~$v$ lies in a~$3$-edge with each of~$x$ and~$y$, there are~$u,w\in V_3$ such that~$uvx,wvy\in E$.
				If~$u\neq w$, we directly obtain a copy of~$B$.
				If~$u=w$, we obtain a copy of~$B$ by using a vertex in~$V_1$ (for instance, with~$x$ being the vertex that is contained in both~$3$-edges of this copy of~$B$).
				Therefore,~$v$ cannot lie in a~$3$-edge with both~$x$ and~$y$.
				Combining this with the above, we obtain that
				\begin{align}\label{eq:3degV3up}
					d^{(3)}(v)\leq \vert V(M)\vert/2\,.
				\end{align}
				On the other hand, by~\eqref{eq:minrho}, we have that~$\rho(a)>\frac{2(n-1)}{8}-\frac{11}{8}$ for all~$a\in V_1$, so that~$d^{(3)}(v)>\frac{4(n-1)}{8}-\frac{19}{4}$ by~\eqref{eq:bcktrckV3} and two applications of Claim~\ref{cl:rhodown}.
				Together with~\eqref{eq:3degV3up}, this yields the claim.
			\end{proof}
			
			Now we can finish the proof.
			Since there is a vertex in~$V_1$ by Corollary~\ref{cor:minmaxrho}~\eqref{it:minrho}, the minimum degree condition~\eqref{eq:mindeg} implies that there is some~$vw\in M$.
            Claim~\ref{cl:rhoup} and~\eqref{eq:minrho} together imply that~$d^{(3)}(v)>\frac{3(n-1)}{8}-\frac{17}{8}$.

            However, we can get an upper bound as follows.
            Recall that by Corollary~\ref{cor:minmaxrho}~\eqref{it:notriangle}, the~$2$-edges in~$L_v$ are isolated~$2$-edges and isolated induced~$2$-uniform stars with at least two edges.
            Lemma~\ref{lem:V2vtsonly1star} yields that there is exactly one star with at least two edges, its centre is~$w$ and, by Claim~\ref{cl:posstowlkdwn} and Lemma~\ref{lem:V(M)assshould}, its leaves are precisely the vertices in~$V_1$.
            Hence, the~$2$-edges in~$L_v$ are precisely given by a star with at least two edges and centre~$w$ whose leaves are precisely the vertices in~$V_1$ as well as a matching of isolated~$2$-edges all of whose vertices are in~$V_3$.
            Therefore, we have 
			\begin{align}\label{eq:degV(M)upprbnd}
				d^{(3)}(v)\leq \vert V_1\vert+\frac{\vert V_3\vert}{2}\leq n-\vert V(M)\vert\,.
			\end{align}
			If~$V_3\neq\emptyset$, \eqref{eq:degV(M)upprbnd} and Claim~\ref{cl:V2largeifV3noem} yield~$d^{(3)}(v) \leq \frac{3(n-1)}{8}-\frac{17}{8}$, which contradicts the lower bound above.
			Hence,~$V_3=\emptyset$, which means that~$H$ looks like Example~\ref{ex:extremalbowtie} up to the sizes of the partition classes.
			A simple optimisation and the extremality of~$H$ then yields that~$H$ is the hypergraph in Example~\ref{ex:extremalbowtie}.

    \section{Proof of Theorem~\ref{thm:match}}\label{sec:proofmatch}
        For the lower bound, take an $n$-vertex simplicial complex containing all pairs as~$2$-edges and an intersecting family of~$3$-edges. 
        The total number of edges is~$1+n+\binom{n}{2}+\binom{n-1}{2} = n^2-n+2$.
        
        For the upper bound, suppose~$C$ is sufficiently large.
        We will show by induction on~$n$ that every simplicial complex on~$n$ vertices with at least~$n^2-n+C$ edges contains a copy of~$F$.
        
        The base case of the induction follows simply from~$C$ being large.
        Now assume that for every~$m\leq n-1$ every simplicial complex on~$m$ vertices with at least~$m^2-m+C$ edges contains a copy of~$F$ and let~$H$ be a simplicial complex on~$n$ vertices with at least~$n^2-n+C$ edges.
        Since~$C$ is large, we can assume that~$n$ is large as well. 
        First note that if there is a vertex~$v$ with~$d(v)\leq 2n -3$, then 
        $$e(H-v)\geq n^2-n+C -(2n-3)-1 = (n-1)^2-(n-1)+C\,,$$ so that, by induction,~$H-v$ and, thus,~$H$ contain a copy of~$F$.
        We may therefore assume that~$d(v)\geq 2n-2$ for every~$v\in V(H)$.
        
        Next we want to argue that if~$H$ contains two intersecting~$4$-edges, then it contains a copy of~$F$.
        Let us therefore assume that~$H$ contains an edge~$e=wxyz$.
        Then it is easy to see that if any~$3$-edge intersects~$e$ in exactly one vertex, then~$H$ contains a copy of~$F$.
        In particular, this implies that if any other~$4$-edge intersects~$e$, it does so in exactly three vertices.
        Assume that there is some such~$e'=w'xyz$.
        Then~$w'$ (and~$w$) cannot be contained in any~$3$-edge that contains a vertex outside of~$e\cup e'$.
        But this yields that~$d(w')\leq 1+n-1+6\leq n+6$, a contradiction.
        Thus, no two~$4$-edges intersect, which also implies that there are no edges of size larger than~$4$ and that every vertex is contained in at most one~$4$-edge.
        
        
        We now show that $d^{(3)}(v)\leq 3n$ for  all~$v\in V(H)$.
        Let~$v\in V(H)$ be arbitrary and let~$a\in V(H)$ be such that~$va\in E(H)$.
        We first show that there is a~$3$-edge containing~$a$ but not~$v$. 
        This follows directly if~$a$ is contained in a~$4$-uniform edge, so we may assume that~$d^{(4)}(v)=0$.
        Therefore, we have
        $$d^{(3)}(a) = d^{(\geq2)}(a)-d^{(2)}(a)\geq 2n-2-(n-1)\geq n-1\,,$$ 
        so there must be a~$3$-edge~$e$ containing~$a$ but not~$v$.
        But then every~$3$-edge that contains~$v$ has to intersect~$e$, so that~$d^{(3)}(v)\leq 3n$.
                
        We now fix $v \in V(H)$ and let~$\alpha = \frac{|N(v)|}{n}<1$.
        Consider some~$3$-edge~$uvw$ and note that every~$3$-edge of~$H$ which intersects~$N(v)$ has to intersect~$uvw$.
        Thus, the above bound on the~$3$-degree applied to~$u$, $v$, and $w$ in turn yields that there are at most~$9n$ such edges.
        Therefore, 
        \begin{align*}
            e(H\!-\!N(v)) 
            \!&\geq
            n^2-n+C -\Big( 2\alpha n + \binom{\alpha n}{2}+ \alpha(1-\alpha) n^2 +9n\Big)\\
            \!&\geq\! \Big(1-\alpha(1-\alpha)-\frac{\alpha^2}{2}\Big)n^2 - (1+2\alpha +9)n +C 
            \!\geq\!
            \big((1-\alpha)n\big)^2 - (1-\alpha)n+C.
        \end{align*}
        The result therefore follows by induction.
	
     \section{Proof of Theorem~\ref{thm:2connvtxgeneral}}\label{sec:gen}
    We will actually prove the following more general theorem, of which Theorem~\ref{thm:2connvtxgeneral} is a special case.
        
    \begin{theorem}\label{thm:2connvtxgeneral-2}
        For every proper~$3$-graph~$F$,~$A\subseteq V(F)$ and~$\varepsilon>0$, there is~$n_0 \in \mathbb{N}$ such that the following holds.
	    If a simplicial complex~$H$ on~$n\geq n_0$ vertices satisfies
	    \begin{align}\label{eq:condonmaxdeg}
	        \frac{e^3 n^{n-2}}{(n+3)^n}\big[\ex^{(3)}(n,F)+n^2\big]\geq\Big[\frac{(\vert A\vert-1)}{\vert A\vert}+\frac{1}{n}\frac{v(F)-\vert A\vert}{\vert A\vert}\Big]\big[1+\sum_{i\geq 2}\frac{\Delta^{(i+1)}}{i}\big]\,,
	   \end{align}
	   where~$\Delta^{(i)}=\max_{xy\in V(H)^{(2)}}d^{(i)}(xy)$,  and~$e(H)\geq(1+\varepsilon)\big[\ex^{(3)}(n,F)+\binom{n}{2}\big]$, then~$H$ contains a copy of~$F^+$ if either:
	   \begin{enumerate}
	       \item\label{it:nolargeedge} $H$ does not contain any edge of size~$4$ or
	       \item\label{it:Fin4edge} $F$ is contained in a~$4$-edge.
	   \end{enumerate}
	\end{theorem}

    To see that Theorem~\ref{thm:2connvtxgeneral-2} indeed implies Theorem~\ref{thm:2connvtxgeneral}, observe that $\big(\tfrac{n}{n+3}\big)^n$ is a decreasing function and tends to $e^{-3}$ as~$n$ tends to infinity.
    Therefore, we have that \eqref{eq:condonmaxdeg1} yields \eqref{eq:condonmaxdeg}.
 
    \begin{proof}[Proof of Theorem~\ref{thm:2connvtxgeneral-2}]
    	Let~$F$ be a proper~$3$-graph,~$A\subseteq V(F)$ and~$\varepsilon>0$.
    	Let~$C$ be a sufficiently large constant ($C^{-1}\ll\varepsilon,1/v(F)$).
    	We will show by induction on~$n$ that if~$H$ is a simplicial complex on~$n\geq v(F^+)$ vertices satisfying~\eqref{eq:condonmaxdeg} and~$e(H)\geq(1+\varepsilon/2)\big[\ex^{(3)}(n,F)+\binom{n}{2}\big]+C$, then~$H$ contains a copy of~$F^+$.
    	It then follows that there is some~$n_0$ such that every simplicial complex~$H$ on~$n\geq n_0$ vertices satisfying~\eqref{eq:condonmaxdeg} and~$e(H)\geq(1+\varepsilon)\big[\ex^{(3)}(n,F)+\binom{n}{2}\big]$ contains a copy of~$F^+$.

     \allowdisplaybreaks

        The base of the induction and the statement for small~$n$ both follow from~$C$ being large.
    	Now let~$n$ be large and assume that for every~$n'\leq n-1$, every simplicial complex~$H'$ on~$n'$ vertices satisfying~\eqref{eq:condonmaxdeg} (with~$n'$ in place of~$n$) and~$e(H')\geq(1+\varepsilon/2)\big[\ex^{(3)}(n',F)+\binom{n'}{2}\big] + C$ contains a copy of~$F^+$.
    	Let~$H$ be a simplicial complex on~$n\geq n_0$ vertices satisfying~\eqref{eq:condonmaxdeg} with~$e(H)\geq(1+\varepsilon/2)\big[\ex^{(3)}(n,F)+\binom{n}{2}\big]+C$.
    	Assume for the sake of contradiction that~$H$ does not contain a copy of~$F^+$.
    	We first derive a minimum degree condition.
    	Suppose that there is a vertex~$v$ in~$H$ with~$d(v)\leq \frac{e^3 n^n}{(n+3)^n}\Big[\frac{\ex^{(3)}(n,F)}{n}+n\Big]$.
        Note that for some fixed~$\eta>0$ and large~$n$, we have~$\frac{(n+3)^n}{e^3 n^n}(1+\varepsilon/2)>1+\eta$.
        By the definition of proper (applied with~$\eta$ instead of~$\varepsilon$) and since~$n$ is large, we therefore have
        \begin{align*}
            \frac{(1+\varepsilon/2)(n+3)^n}{e^3 n^n}\big[\ex^{(3)}(n,F)-\ex^{(3)}(n-1,F)\big]\geq\frac{\ex^{(3)}(n,F)}{n}
        \end{align*}
        or, after rearranging,
        \begin{align*}
            (1+\varepsilon/2)\big[\ex^{(3)}(n,F)-\ex^{(3)}(n-1,F)\big]\geq\frac{e^3 n^n}{(n+3)^n}\frac{\ex^{(3)}(n,F)}{n}\,.
        \end{align*}
        Keeping in mind that~$n$ is large and so~$\frac{e^3 n^n}{(n+3)^n}<1+\frac{\eps}{4}$, this in turn implies that
    	\begin{align*}
            e(H\!-\!v) &\!\geq\!\Big(1+\frac{\eps}{2}\Big)\big[\ex^{(3)}(n,F)+\binom{n}{2}\big]+C-\frac{e^3 n^n}{(n+3)^n}\Big[\frac{\ex^{(3)}(n,F)}{n}+n\Big]-1\\
            &\!\geq\! \Big(1+\frac{\eps}{2}\Big)\big[\ex^{(3)}(n-1,F)+\binom{n-1}{2}\big]+C\,.
        \end{align*}
        In order to apply induction, we must also check that~\eqref{eq:condonmaxdeg} holds for~$H-v$ and~$n-1$ instead of~$H$ and~$n$.
        But, by a double counting argument from~\cite{KNS:64}, we have ~$\frac{\ex^{(3)}(n-1,F)}{n-3}\geq\frac{\ex^{(3)}(n,F)}{n}$, so, 
        since~$n$ is large, we have
        \begin{align*}
            &\frac{e^3 (n-1)^{n-3}}{(n+2)^{n-1}}\big[\ex^{(3)}(n-1,F)+(n-1)^2\big]\\
            &=\frac{e^3 (n-1)^{n-3}(n-3)}{(n+2)^{n-1}}\frac{\ex^{(3)}(n-1,F)}{n-3}+\frac{e^3 (n-1)^{n-1}(n+3)^n}{(n+2)^{n-1}e^3 n^n}\frac{e^3 n^n}{(n+3)^n}\\
            &\geq \frac{e^3 (n-1)^{n-3}(n-3)}{(n+2)^{n-1}}\frac{\ex^{(3)}(n,F)}{n}+\frac{n}{n-1}\frac{e^3 n^n}{(n+3)^n}\\
            &\geq \frac{n}{n-1}\frac{e^3 n^{n-2}}{(n+3)^n}\big[\ex^{(3)}(n,F)+n^2\big]\\
            &\geq \frac{n}{n-1}\Big[\frac{(\vert A\vert-1)}{\vert A\vert}+\frac{1}{n}\frac{v(F)-\vert A\vert}{\vert A\vert}\Big]\big[1+\sum_{i\geq 2}\frac{\Delta_H^{(i+1)}}{i}\big]\\
            &\geq \Big[\frac{(\vert A\vert-1)}{\vert A\vert}+\frac{1}{n-1}\frac{v(F)-\vert A\vert}{\vert A\vert}\Big]\big[1+\sum_{i\geq 2}\frac{\Delta_{H-v}^{(i+1)}}{i}\big]\,,
        \end{align*}
        where the second inequality follows from $\frac{e^3 (n-1)^{n-3}(n-3)}{(n+2)^{n-1}}\geq \frac{n}{n-1}\frac{e^3 n^{n-2}}{(n+3)^n}$ for large~$n$.
        Hence, we can apply induction to~$H-v$, yielding~$F^+\subseteq H-v\subseteq H$, a contradiction.
    	Thus, we may assume that~$d(v)>\frac{e^3 n^n}{(n+3)^n}\big[\frac{\ex^{(3)}(n,F)}{n}+n\big]$ for every~$v\in V(H)$.

        If either of the conditions~\eqref{it:nolargeedge} or~\eqref{it:Fin4edge} holds, we know that~$H$ contains a copy of~$F$.
    	Abusing notation slightly, we now use~$F$ and~$A$ to denote the images of~$F$ and~$A$ under some embedding.
    	For~$H$ not to contain a copy of~$F^+$, there must be a missing~$2$-edge between every vertex in~$V(H-F)$ and some vertex in~$A$.
    	Therefore, there is some vertex~$a\in A$ with at least~$\frac{n- v(F)}{\vert A\vert}$ missing~$2$-edges.
    	Note also that since~$d^{(i)}(xy)\leq\Delta^{(i)}$ for all~$xy\in V(H)^{(2)}$ and~$H$ is hereditary, we have that, for every~$v\in V$,
    	$$d(v)\leq d^{(2)}(v)\big[1+\sum_{i\geq 2}\frac{\Delta^{(i+1)}}{i}\big]\,.$$
    	Combining this with our lower bound on the degree (of any vertex) and our upper bound on~$d^{(2)}(a)$, we get that
    	\begin{align*}
    	    \frac{e^3 n^n}{(n+3)^n}\big[\frac{\ex^{(3)}(n,F)}{n}+n\big]&<d(a)\leq d^{(2)}(a)\big[1+\sum_{i\geq 2}\frac{\Delta^{(i+1)}}{i}\big]\\
    	    &\leq \big(n-1-\frac{n- v(F)}{\vert A\vert}\big)\big[1+\sum_{i\geq 2}\frac{\Delta^{(i+1)}}{i}\big]\\
    	    &\leq \frac{(\vert A\vert-1)n+ v(F)-\vert A\vert}{\vert A\vert}\big[1+\sum_{i\geq 2}\frac{\Delta^{(i+1)}}{i}\big]\,
    	\end{align*}
    	and, hence,
    	\begin{align*}
    	    &\frac{e^3 n^n}{(n+3)^n}\big[\frac{\ex^{(3)}(n,F)}{n^2}+1\big]<\Big[\frac{(\vert A\vert-1)}{\vert A\vert}+\frac{1}{n}\frac{v(F)-\vert A\vert}{\vert A\vert}\Big]\big[1+\sum_{i\geq 2}\frac{\Delta^{(i+1)}}{i}\big]\,.
    	\end{align*}
    	But this contradicts~\eqref{eq:condonmaxdeg}, meaning that~$H$ does contain a copy of~$F^+$.
    \end{proof}

 \section{Proof of Theorem~\ref{thm:Bollobas}}\label{sec:proofBollobas}
	    \begin{proof}[Proof of Theorem~\ref{thm:Bollobas}]
    	    For the lower bound, simply consider the simplicial complex on~$n$ vertices that has three (almost) balanced partition classes and all~$2$-edges and~$3$-edges which intersect each partition class in at most one vertex.
    	    
    	    For the upper bound, first note that if a simplicial complex contains an edge of size at least four, then it contains a copy of~$F$.
    	    Thus, we only need to show that given~$\varepsilon>0$, there is some~$n_0$ such that every simplicial complex~$H$ on~$n\geq n_0$ vertices with at least~$(\frac{2}{9}+\varepsilon)\binom{n}{3}$ edges, all of which are of size at most three, contains a copy of~$F$.
    	    Next, note that~$F$ is contained in the simplicial complex $$F^+=(\{v_1,v_2,v_3,v_4,v_5\},\{v_1v_2v_3,v_2v_3v_4,v_1v_4v_5\})\,.$$
    	    Bollob\'as~\cite{B:74} showed that for every~$\varepsilon>0$, there is some~$n_1$ such that every~$3$-graph on~$n\geq n_1$ vertices with at least~$(\frac{2}{9}+\frac{\varepsilon}{2})\binom{n}{3}$ edges contains the~$3$-graph $$F_3=(\{v_1,v_2,v_3,v_4,v_5\},\{v_1v_2v_3,v_2v_3v_4,v_1v_4v_5\})\,.$$
    	    From this, we can infer our result as follows.
            For every~$\varepsilon>0$, there is some~$n_2$ such that every simplicial complex~$H$ on~$n\geq n_2$ vertices with at least~$(\frac{2}{9}+\varepsilon)\binom{n}{3}$ edges not containing any edge of size larger than three has at least~$(\frac{2}{9}+\frac{\varepsilon}{2})\binom{n}{3}$ $3$-edges.
    	    Thus, the~$3$-uniform layer of any simplicial complex~$H$ on~$n\geq n_0=\max\{n_1,n_2\}$ vertices with at least~$(\frac{2}{9}+\varepsilon)\binom{n}{3}$ edges contains a copy of~$F_3$, which in the simplicial complex~$H$ induces a copy of~$F^+$ and, thereby, a copy of~$F$.
	    \end{proof}

    \section{Concluding Remarks}
    
    In this work, we studied the Tur\'an problem for simplicial complexes. 
    Since this problem seems to be of comparable difficulty to the Tur\'an problem for uniform hypergraphs, any non-trivial result on the following problem would be of interest.

    \medskip

    \begin{problem}
        Given a simplicial complex~$F$ and a positive integer~$n$, determine~$\ex(n,F)$ or its asymptotics.
    \end{problem}

    \medskip

    Theorem~\ref{thm:bowtie} shows that this problem is distinct from the problem for uniform hypergraphs in that, for some~$F$, the extremal construction for~$F$ has more than one incomplete layer.
    On the other hand, we showed that for several simplicial complexes~$F$, the extremal number is asymptotically given by the trivial bound~$\ex^{(k)}(n,F_{k})+\sum_{i=0}^{k-1}\binom{n}{i}$, where $k = \dim(F)+1$. 
    This suggests the following problem.

    \medskip

    \begin{problem}
        Given an integer~$k\geq 3$, characterise the~$(k-1)$-dimensional simplicial complexes~$F$ with~$\ex(n,F)=\ex^{(k)}(n,F_{k})+\sum_{i=0}^{k-1}\binom{n}{i}$.
    \end{problem}

    \medskip

    As mentioned in the introduction, Hajnal (see~\cite{B:72, FS:05}) initiated the investigation of a variant of the Brown--Erd\H{o}s--S\'os problem~\cite{SEB:73} for simplicial complexes. We believe this special case of the Tur\'an problem for simplicial complexes to be of particular interest.
    Recall that~$\mathcal{F}(k,d)$ denotes the family of all simplicial complexes with vertex set~$[k]$ and~$1+\sum_{i=0}^d\binom{k}{i}$ edges.

    \medskip

    \begin{problem}\label{prob:F}
        For~$k\geq 5$ and~$d\geq 2$, determine~$\ex(n,\mathcal{F}(k,d))$ or its asymptotics.
    \end{problem}
    
    \medskip

    We have already noted that Bollob\'as and Radcliffe \cite{BR:95} solved Problem~\ref{prob:F} for~$k=4$ and~$d=2$.
    A simple inductive argument gives the following upper bound for the case~$k=5$ and~$d=2$.

    \begin{prop} \label{prop:52}
    For every~$n\geq 5$, $\ex(n,\mathcal F(5,2))\leq \tfrac{1}{2}(n^{5/2} + 3n^2)$. 
    \end{prop}

    \begin{proof}
        We proceed by induction, noting that the base case~$n=5$ follows easily since~$\tfrac{1}{2}(n^{5/2} + 3n^2) \geq 55> 32 = \sum_{i=0}^{5}\binom{5}{i}$.
        Suppose now that~$H$ is a simplicial complex on~$n$ vertices with~$e(H)>\tfrac{1}{2}(n^{5/2} + 3n^2)$.
       
        If there is a vertex~$v$ such that~$d(v)\leq \tfrac{1}{2}(n^{3/2} + 3n-2)$, then 
        $$e(H-v) > \frac{1}{2}(n^{5/2} + 3n^{2} - n^{3/2} - 3n + 2) -1 \geq  \frac{1}{2}\big((n-1)^{5/2} + 3(n-1)^{2}\big), $$
        so, by induction, $H-v$ contains some $F\in \mathcal F(5,2)$. 
        Therefore, we may assume that
        \begin{align}\label{eq:C4mindeg}
            d(v)>\tfrac{1}{2}(n^{3/2} + 3n-2)
        \end{align}
        for every vertex~$v\in V(H)$. 
        
        Suppose that there is an edge $e\in E^{(4)}(H)$. 
        Then it is easy to see that for any vertex $v\in V(H)\setminus e$ the set~$e\cup\{v\}$ spans at least~$17$ edges. 
        Hence, we may assume that~$E^{(4)}=\emptyset$ and so \eqref{eq:C4mindeg} yields that~$d^{(3)}(v) \geq \tfrac{1}{2}(n^{3/2} + n)$ for every vertex~$v\in V(H)$.
        Therefore, for any vertex~$v\in V(H)$, the classical theorem of K\H ovari, S\'os and Tur\'an \cite{KST} implies that there is a ($2$-uniform) copy of~$C_4$ in the link of~$v$. 
        It is not hard to check that the vertices of such a~$C_4$ together with~$v$ span at least~$18$ edges.         
    \end{proof}

    \begin{question}\label{quest:5,17}
        Is it true that~$\ex(n, \mathcal F(5,2)) =\Omega(n^{5/2})$?
    \end{question}

    \medskip
    
    Let~$LC_4^{(3)}$ be the~$3$-uniform hypergraph with~$V(LC_4^{(3)}) = \{u,v_1,\dots, v_4\}$ and~$E(LC_4^{(3)})=\{uv_iv_{i+1}\colon i\in[4]\}$, where the indices are viewed as in~$\mathds{Z}/4\mathds{Z}$.
    In other words,~$LC_4^{(3)}$ is the~$3$-uniform hypergraph consisting of a vertex~$u$ with a ($2$-uniform) $C_4$ in its link. 
    By following the proof of Proposition~\ref{prop:52}, we get that~$\ex(n, \mathcal F(5,2))\leq \ex^{(3)}\big(n,LC_4^{(3)}\big) + \binom{n}{2}+n+1$. 
    In particular, improving the bound $\ex^{(3)}\big(n,LC_4^{(3)}\big) = O(n^{5/2})$ would answer Question~\ref{quest:5,17} in the negative. 
    However, a result of Brown, Erd\H{o}s and S\'os~\cite[Theorem~3]{SEB:73} shows that $\ex^{(3)}\big(n,LC_4^{(3)}\big) = \Omega(n^{5/2})$, perhaps giving some evidence towards a positive answer for Question~\ref{quest:5,17}.
    
    




    \medskip
    
    Let us conclude by briefly discussing the problem of determining the extremal number of the simplicial complex induced by the~$k$-uniform clique on~$t$ vertices for~$t\geq k+2$.
    Denote the simplicial complex given by the downward closure of~$K_t^{(k)}$ by~$K(t,k)$.
    
    \begin{problem}\label{prob:cliques}
        Given~$t\geq k+2$, determine the asymptotics of~$\ex(n,K(t,k))$.
    \end{problem}

    A simple symmetrisation argument yields the following result for the case~$k=2$.
    \begin{observation}\label{obs:graphcliques}
        For~$t\geq 3$,~$\varepsilon>0$ and~$n$ sufficiently large, $\ex(n,K(t+1,2)) =(1\pm\varepsilon) \frac{t!}{t^{t}}\binom{n}{{t}}$.
    \end{observation}

    \begin{proof}
        Let~$H$ be a simplicial complex on~$n$ vertices without a~$K(t+1,2)$ and, subject to this condition, with the maximum number of edges.
        Assume that there are three vertices $x$, $y$ and $z$ such that $xy$ and $yz$ are not in $E(H)$ but $xz$ is.
        If~$d(y)\geq d(x),d(z)$, then we can replace $x$ and $z$ by copies of $y$.
        This would yield a simplicial complex without a~$K(t+1,2)$ but with more edges than~$H$, a contradiction.
        Otherwise, for one of~$x$ and~$z$, say, without loss of generality, for~$x$, we have~$d(x)> d(y)$.
        In this case we can replace $y$ by a copy of $x$, again yielding a simplicial complex without a~$K(t+1,2)$ but with more edges than~$H$.
        Thus, we may assume that non-adjacency is an equivalence relation, so that~$H$ is complete~$r$-partite for some~$r\leq t$, that is, the vertex set can be partitioned into~$r$ sets such that every edge contains at most one vertex from each of these sets.
        Therefore,~$H$ has~$\max_{r\in[t]}(1\pm\varepsilon)\frac{r!}{r^r}\binom{n}{r}=(1\pm\varepsilon)\frac{t!}{t^t}\binom{n}{t}$ edges.
    \end{proof}

    This problem is related to another much-studied problem in extremal combinatorics.
    Following Alon and Shikhelman~\cite{AS:16}, given two~$k$-graphs~$F$ and~$G$, we write~$\ex(n,F,G)$ for the maximum number of copies of~$F$ in a $G$-free~$k$-graph on~$n$ vertices.
    
    \begin{observation}\label{obs:cliques}
        For~$t\geq k+2$, $\varepsilon>0$ and $n$ sufficiently large, 
        $$\ex(n,K(t,k)) = (1\pm\varepsilon)\ex(n,K_{t-1}^{(k)}, K_{t}^{(k)})\,.$$
    \end{observation}
    
    \begin{proof}
        Consider a simplicial complex~$H$ on~$n$ vertices with at least~$(1+\varepsilon)\ex(n,K_{t-1}^{(k)}, K_{t}^{(k)})$ edges.
        Any~$t$-edge in~$H$ would immediately yield a~$K(t,k)$, so we can assume that~$H$ does not contain any edges of size larger than~$t-1$.
        Thus, since~$n$ is taken sufficiently large, $H$ contains at least~$(1+\varepsilon/2)\ex(n,K_{t-1}^{(k)}, K_{t}^{(k)})$ edges of size~$t-1$.
        Each of these induces a~$K_{t-1}^{(k)}$.
        But, by definition, this implies that~$H$ contains a~$K_t^{(k)}$ and, thereby, a~$K(t,k)$.

        Suppose now that~$H''$ is a~$k$-uniform hypergraph on~$n$ vertices that maximises the number of copies of~$K_{t-1}^{(k)}$ without containing a~$K_t^{(k)}$.
        Let~$H'$ be the~$(t-1)$-uniform hypergraph on the same vertex set that is obtained by placing a~$(t-1)$-edge for every~$K_{t-1}^{(k)}$ of~$H''$
        and let~$H$ be the 
        downward closure of~$H'$.
        Clearly, since~$n$ is large,~$H$ has more than~$\ex(n,K_{t-1}^{(k)},K_t^{(k)})$ edges.
        It is also easy to check that~$H$ does not contain a copy of~$K(t,k)$. 
    \end{proof}

    Combining Observation~\ref{obs:cliques} with~\cite[Proposition~2.2]{AS:16}, which implies that $\ex(n,K_{t-1}, K_{t}) = (1 + o(1)) (\frac nt)^t$, gives another proof of Observation~\ref{obs:graphcliques}. 
    Moreover, using flag algebras, Bodnar~\cite{bodnar} recently proved that~$\ex(n,K_4^{(3)},K_5^{(3)})=(\frac{3}{8} +o(1))\binom{n}{4}$.
    By Observation~\ref{obs:cliques}, this implies that~$\ex(n,K(5,3))=(\frac{3}{8}+o(1))\binom{n}{4}$.
    To the best of our knowledge, Problem~\ref{prob:cliques} remains open for all other $(t, k)$ with~$t\geq 6$ or~$k\geq 3$.

\end{document}